\documentclass[a4paper,11pt]{amsart}
\usepackage{amsfonts}
\usepackage{amssymb}
\usepackage[utf8]{inputenc}
\usepackage{amsmath}
\usepackage{pdflscape}
\usepackage{float}
\usepackage{easyReview}
\usepackage{graphicx}
\usepackage[colorlinks=true, linkcolor=blue, citecolor=red]{hyperref}

\newtheorem{theorem}{Theorem}[section]
\newtheorem{lemma}[theorem]{Lemma}
\newtheorem{proposition}[theorem]{Proposition}
\newtheorem{corollary}[theorem]{Corollary}
\theoremstyle{remark}
\newtheorem{remark}[theorem]{Remark}
\numberwithin{equation}{section}
\newcommand{\T}{\mathbb T}
\newcommand{\R}{\mathbb R}
\newcommand{\C}{\mathbb C}
\newcommand{\Z}{\mathbb Z}
\newcommand{\supp}{\operatorname{supp}}
\newcommand{\Ran}{\operatorname{Ran}}

\numberwithin{equation}{section}
\makeatletter
\@namedef{subjclassname@2010}{\textup{2020} Mathematics Subject Classification}
\makeatother

\begin{document}
	
	\pagenumbering{arabic}	
\title[UCP for the periodic BBM equation]{Unique continuation for the periodic Benjamin--Bona--Mahony equation}
\author[Capistrano-Filho]{Roberto de A. Capistrano-Filho*}
	\address{Departamento de Matem\'atica,  Universidade Federal de Pernambuco (UFPE), 50740-545, Recife (PE), Brazil.}
	\email{roberto.capistranofilho@ufpe.br}
	
	\author[Gonzalez Martinez]{Victor Hugo Gonzalez Martinez}
	\address{Departamento de Matem\'atica,  Universidade Federal de Pernambuco (UFPE), 50740-545, Recife (PE), Brazil.}
		\email{victor.martinez@ufpe.br}
	\address{Laboratoire Analyse, Géométrie et Applications, CNRS UMR 7539, Université Sorbonne Paris-Nord, F-93430 Villetaneuse, France.}
\email{gonzalezmartinez@math.univ-paris13.fr}
	
	\author[Nascimento]{Ailton Campos do Nascimento}
	\address{Departamento de Matem\'atica, Universidade Federal do Piau\'i (UFPI),64049-550, Teresina (PI), Brazil.}
	\email{ailton.nascimento@ufpi.edu.br}

\date{\today}
		\thanks{*Corresponding author: \url{roberto.capistranofilho@ufpe.br}}
		\thanks{\textbf{Funding.} Capistrano–Filho was partially supported by the CAPES-COFECUB program grant number 88887.879175/2023-00, CNPq grant numbers 301744/2025-4 and 421573/2023-6, and PROPG (UFPE). V. H. Gonzalez Martinez was supported by CAPES/COFECUB grant 88887.879175/2023-00 and CNPq grants 421573/2023-6, 402652/2024-0 and 300997/2025-6.}

\begin{abstract}
We prove unique continuation for the real periodic Benjamin--Bona--Mahony
equation in the energy space: a solution that equals a constant on a nonempty
space--time strip is that constant everywhere. In particular, Rosier's unique
continuation conjecture \cite{Rosier2016} holds without any smallness or mean
assumption. It even suffices to observe, at one point $x_0$ and for $t$ in an
interval, the solution and the potential $\psi=(1-\partial_x^2)^{-1}(u+u^2/2)$
of its flux: $u(x_0,t)=C$ and $\psi(x_0,t)=C+C^2/2$ imply $u\equiv C$, and the
second condition cannot be dropped when $C=-1$. The proof uses an exponentially
weighted Hamiltonian of indefinite sign with a coercive dissipation identity,
and it applies to every nonconstant real polynomial flux. Consequently,
Rosier's stabilization results for two localized feedback laws and for a
boundary feedback hold unconditionally, for arbitrary data. Stationary step
profiles show that the theorem fails in $H^s$ for $0\leq s<1/2$. We also give a
generating-series proof of single-point unique continuation for the linearized
equation and discuss related control results.
\end{abstract}

\maketitle
\section{Introduction}
On $\T=\R/(2\pi\Z)$ we study the real Benjamin--Bona--Mahony (BBM)
equation, introduced by Peregrine \cite{Peregrine} and by Benjamin, Bona and
Mahony \cite{BBM} as a regularized model for unidirectional long waves,
\begin{equation}\label{eq:bbm}
 (1-\partial_x^2)u_t+\partial_x\left(u+\frac{u^2}{2}\right)=0.
\end{equation}
For initial data in $H^1(\T)$ it has a global, two-sided flow and conserves
$$\|u\|_{H^1}^2=\int_\T(u^2+u_x^2)\,dx.$$ On the line, the Cauchy problem is
globally well posed in $H^s(\R)$ exactly for $s\geq0$ \cite{BonaTzvetkov}, and on
the torus it is globally well posed in $H^s(\T)$, $s\geq0$
\cite{Roumegoux,RosierZhang}. Although its solutions are analytic in time,
they need not be analytic in space. In particular, a local space--time
observation does not lead to unique continuation by a direct analytic
identity theorem in both variables.

Unique continuation from an open set is a basic tool in the control theory
of nonlinear dispersive equations. Through the compactness--uniqueness
method, stabilization under localized damping and global control results
are reduced to showing that a solution which vanishes, or is constant, on a
space--time strip is trivial. For the Korteweg--de Vries (KdV) equation,
this property was obtained by Saut and Scheurer \cite{SautScheurer} with
Carleman estimates and by Zhang \cite{Zhang92} through inverse scattering.
Related uniqueness results from the support of solutions are due to
Bourgain \cite{Bourgain} and to Kenig, Ponce and Vega \cite{KPV02,KPV03}. On
periodic domains, continuation is combined with propagation of regularity
and smoothing effects in the control and stabilization theory of KdV
\cite{RussellZhang,LaurentRosierZhang}, the Benjamin--Ono equation
\cite{LaurentLinaresRosier}, and higher-order or coupled dispersive models
\cite{CapistranoKwakVielma,CapistranoGomes}. For the nonlinear Schr\"odinger
equation, see \cite{DehmanGerardLebeau}. The BBM equation differs in a
structural way. Its dispersion relation $\xi/(1+\xi^2)$ is bounded, and the
free group $S(t)$ of \eqref{eq:linbbm} below is unitary on every $H^s(\T)$
and satisfies $S(t)f-f\in H^{s+1}(\T)$ for $f\in H^s(\T)$. Singularities
therefore neither travel nor disappear: at high frequencies, no information
moves from an observation region to the rest of the torus. Smoothing and
propagation-of-regularity arguments are therefore unavailable. The same
feature makes the spectrum of the linearized operator accumulate at $0$.

For control of BBM-type equations, the accumulation of the spectrum was
identified by Micu \cite{Micu} as an obstruction to exact controllability.
Zhang and Zuazua \cite{ZhangZuazua} studied unique continuation for the
linearized equation with a space-dependent potential (see also
\cite{Yamamoto}). On the torus, Rosier
and Zhang \cite{RosierZhang} overcame the obstruction with moving controls.
They obtained local and large-time global exact controllability, and
semiglobal exponential stabilization. Moving controls have since been used
for the regularized Benjamin--Ono equation \cite{MontesCordoba}. For the
$abcd$ family of Boussinesq systems on a periodic domain, see
\cite{MORZ}. Rosier \cite{Rosier2016} studied the BBM equation with
localized damping and reduced its asymptotic stability to a continuation
property discussed below. Unique continuation for the KP--BBM-II equation is
proved in \cite{Mammeri}. Continuation and decay under localized damping for
other BBM-type models appear in \cite{PazotoSotoVieira}.

Rosier and Zhang proved in \cite[Theorem~3.1]{RosierZhang} that an $H^1$
solution vanishing on a nonempty strip is zero if its initial mean is
nonnegative and $\|u_0\|_{L^\infty}<3$. For the generalized equations
$$u_t-u_{txx}+f(u)_x=0$$ without drift term, where $f\in C^1(\R)$ is
nonnegative and $0$ is its only zero near the origin, they established
continuation for arbitrary $H^1$ data \cite[Theorem~4.1]{RosierZhang}. The latter mechanism also applies to
the special constant level $-1$ of \eqref{eq:bbm}, since
$u+u^2/2=(u+1)^2/2-1/2$, the function $u+1$ solves an equation with the
one-signed flux $v^2/2$. A related one-sign flux argument, for $H^4$ data
and also on the line, is given in \cite{daSilvaFreire}, which also shows
that such equations have no nontrivial compactly supported solutions. These restrictions leave the usual BBM
equation at the zero level unresolved. In his study of localized damping,
Rosier \cite{Rosier2016} explicitly posed the following problem.

\medskip\noindent\textbf{Rosier's UCP conjecture \cite{Rosier2016}.}
\emph{There exists $\delta>0$ such that, for every
$v_0\in H^1(\T)$ with $\|v_0\|_{H^1}<\delta$, if the solution of
$$
\begin{cases}v_t-v_{xxt}+v_x+vv_x=0,\\
v(0)=v_0,
\end{cases}
$$ satisfies $v=0$ on
$\omega\times(0,T)$ for a nonempty open $\omega\subset\T$ and some
$T>0$, then $v_0=0$ (and hence $v\equiv0$).}
\medskip

The conjecture was still described as open in
\cite[Section~1]{SierraFonsecaPazoto}, where the approach of \cite{Rosier2016}
is adapted to a linear system coupling two higher-order BBM-type equations
(see also \cite{BautistaPazoto,MicuPazoto} for related Boussinesq systems).
The conjecture imposes no sign on the mean. Our result proves its conclusion
for all real $H^1$ data and, more generally, for observation at any constant
level.

\begin{theorem}\label{thm:main}
Let $u\in C(\R;H^1(\T;\R))$ be a solution of \eqref{eq:bbm}. If $u=C$ on $\omega\times(0,T)$ for some $C\in\R$, some nonempty open $\omega\subset\T$, and some $T>0$, then $u\equiv C$ on $\T\times\R$.
\end{theorem}

In particular, Rosier's conjecture holds with no restriction on $\delta$. Neither the size of the datum nor the sign of its mean enters the conclusion. To our knowledge, $C=-1$ was the only level previously covered without restrictions on the data. For every other level, including the level $C=0$ of Rosier's conjecture, the flux $F_\beta$ introduced below changes sign, and arguments based on its positivity do not apply. The constant-value formulation will also identify the limiting states for derivative damping.

The observation in Theorem~\ref{thm:main} can be reduced to two scalar functions of time at a single point. For a solution $u$ of \eqref{eq:bbm} we write
\begin{equation}\label{eq:defpsi}
\begin{cases}
 q&=u+\frac{u^2}{2},\\ 
 \psi&=(1-\partial_x^2)^{-1}q,
 \end{cases}
\end{equation}
so that \eqref{eq:bbm} reads $u_t=-\psi_x$. Explicitly, $$\psi(x,t)=\int_\T G(x-y)\,q(y,t)\,dy,$$ where $$G(x)=\frac{\cosh(\pi-|x|)}{2\sinh\pi}$$ for $|x|\leq\pi$. Since $$\frac{d}{dt}\int_{x_0}^{x}u(y,t)\,dy=\psi(x_0,t)-\psi(x,t),$$ the function $\psi(x_0,\cdot)$ measures the flux of mass through $x_0$, so it is not determined by the values of $u$ near $x_0$.

\begin{theorem}\label{thm:point}
Let $u\in C(\R;H^1(\T;\R))$ be a solution of \eqref{eq:bbm}, and let $\psi$ be given by \eqref{eq:defpsi}. If, for some $x_0\in\T$, $C\in\R$ and $T>0$,
\begin{equation}\label{eq:pointobs}
 u(x_0,t)=C\qquad\text{and}\qquad \psi(x_0,t)=C+\frac{C^2}{2}\qquad\text{for all }t\in(0,T),
\end{equation}
then $u\equiv C$ on $\T\times\R$.
\end{theorem}

Here $u(\cdot,t)$ and $\psi(\cdot,t)$ are the continuous representatives. A solution which equals $C$ on $\omega\times(0,T)$ satisfies \eqref{eq:pointobs} at every point of $\omega$ (Lemma~\ref{lem:strip}), so Theorem~\ref{thm:point} contains Theorem~\ref{thm:main}. The second condition in \eqref{eq:pointobs} cannot be dropped in general: the equation is invariant under $$u(x,t)\mapsto-2-u(2x_0-x,t),$$ so that every solution with $1+u_0$ odd with respect to $x_0$ and not identically zero satisfies $u(x_0,t)=-1$ for all $t$ without being constant (Remark~\ref{rem:level}). The first condition cannot be dropped either when $C\neq-1$: the constant solution $u\equiv-2-C$ satisfies $\psi\equiv C+C^2/2$. At the level $C=-1$, the second condition alone suffices (Remark~\ref{rem:level}). For the linearized equation, by contrast, the observation of the solution alone at one point suffices (Proposition~\ref{thm:lin}).

Two further consequences of the method are worth recording. First, Proposition~\ref{prop:liouville} below is a Liouville theorem on the line: a real solution in $C(\R;H^1(\R))$ of the shifted equation which is supported in a fixed compact interval during a nonempty open time interval vanishes identically. This extends the compact-support statement of \cite[Theorem~1.2]{daSilvaFreire}, and the corresponding step in the proof of \cite[Theorem~4.1]{RosierZhang}, to the sign-changing fluxes $F_\beta$, $\beta\ne0$. Second, the proof uses the quadratic nonlinearity only through three properties: the vector field is analytic on $H^1$, the flux has an antiderivative, and $0$ is an isolated zero of the shifted flux. Remark~\ref{rem:general} explains why the same conclusion therefore holds for every nonconstant real polynomial flux.

The regularity issue should be distinguished from smallness. Rosier and Zhang \cite[Remark~3.2]{RosierZhang} observed that a stationary step profile of height $-2$ solves \eqref{eq:bbm} in $C(\R;L^2(\T))$, which shows that their Theorem~3.1 fails without the assumptions $u_0\in H^1$ and nonnegative mean. The same construction works on any interval. If $E$ is a nonempty proper interval in $\T$ whose complement has interior, then
\begin{equation}\label{eq:step}
\begin{cases}
& u_E(x,t)=-2\mathbf 1_E(x),\\
 &u_E+u_E^2/2=0\quad\hbox{almost everywhere},
\end{cases}
\end{equation}
is a stationary distributional solution of the \emph{conservative} equation \eqref{eq:bbm}. It vanishes on an open strip and is nonzero. Here the flux derivative is well defined; the product $u_E(u_E)_x$ need not be interpreted separately. These profiles satisfy
\[
 \int_\T u_E\,dx=-2|E|,\qquad
 \|u_E\|_\infty=2,\qquad \|u_E\|_2=2\sqrt{|E|}.
\]
Since $u_E+u_E^2/2=0$, each $u_E$ is a fixed point of the Duhamel formulation, hence it is the solution provided by the well-posedness theory in $H^s(\T)$, $s\geq0$ \cite[Theorem~2.1]{RosierZhang}, not merely a distributional one. Their jumps exclude them from $H^1$, but $\mathbf 1_E\in H^s(\T)$ for every $s<1/2$. Moreover, the Fourier coefficients of $\mathbf 1_E$ are bounded by $\min\{|E|/(2\pi),\,1/(\pi|k|)\}$, whence $$\|u_E\|_{H^s}^2\leq C_s|E|^{1-2s}\to0,$$ as $|E|\to0$. Thus the analogue of Theorem \ref{thm:main} fails in $H^s(\T)$ for every $0\leq s<1/2$, even for data of arbitrarily small norm; we do not know whether it holds for $1/2\leq s<1$. These examples do not disprove an energy-space statement and do not exclude a small-$L^\infty$ theorem with threshold below $2$.

Theorem \ref{thm:main} concerns a time interval. Hong and Ponce \cite{HongPonce} establish continuation from certain one-time observations about the level $-1$ and construct smooth examples at other levels satisfying corresponding one-time conditions. Such observations are different from persistence on a strip. Rosier and Zhang \cite[Section~1]{RosierZhang} point out, following \cite{ZhangZuazua}, that Theorems~3.1--3.4 of \cite{DavilaMenzala}, obtained by Carleman estimates, are not correct without further assumptions; our proof does not use them. Zhang and Zuazua \cite{ZhangZuazua} develop spectral continuation methods for a linearized equation with a spatial potential and Dirichlet boundary conditions. That geometry differs from the periodic, constant-coefficient problem considered here.

The proof starts with two ideas from Rosier and Zhang
\cite{RosierZhang}: analyticity in time propagates a strip condition to
all real times, and cutting the period through the zero strip gives a
solution $z=u-C$ on $\R$ with fixed compact support. Writing
$\beta=1+C$, $F_\beta(z)=\beta z+z^2/2$ and
$w=(1-\partial_x^2)^{-1}F_\beta(z)$, the equation $z_t=-w_x$
forces $w$ to have the same fixed compact support. The conditions
$$\int_\R e^{\pm x}F_\beta(z)\,dx=0$$ follow. For the one-sign case
$\beta=0$, this already proves rigidity; when $\beta\ne0$ these moments
allow cancellation. Our new ingredient is the weighted Hamiltonian
\[
\begin{cases}
 N_\alpha(t)&= \displaystyle\int_\R e^{\alpha x}
 \left(\frac{\beta z^2}{2}+\frac{z^3}{6}\right)dx,
\\
 N_\alpha'(t)&= \displaystyle-\frac\alpha2\left(\|\partial_x(e^{\alpha x/2}w)\|_2^2+
 \left(\frac{\alpha^2}{4}-1\right)\|e^{\alpha x/2}w\|_2^2\right).
 \end{cases}
\]
For $\alpha>2$ the derivative is nonpositive even though the Hamiltonian
density changes sign. Integrating this identity in time gives $$\int_\R\|w(t)\|_2^2\,dt<\infty,$$ and compactness of the uniformly bounded, fixed-support trajectory then supplies zero limits of $N_\alpha$ along sequences in both time directions, the monotonicity forces $N_\alpha\equiv0$, which gives $w=0$ and $z=0$. The full regularity,
support, and compactness justifications are in Section~\ref{sec:weighted}.

Theorem~\ref{thm:point} is proved without passing to the line. On a fundamental domain $[-\pi,\pi]$ whose endpoints are at $x_0$, the same computation for $$\widetilde N_\alpha(t)=\int_{-\pi}^{\pi}e^{\alpha x}(\beta z^2/2+z^3/6)\,dx,$$ with $w=\psi-C-C^2/2$ and $W=e^{\alpha x/2}w$, gives
\[
 \widetilde N_\alpha'(t)=-\frac\alpha2\int_{-\pi}^{\pi}\Big(W_x^2+\Big(\frac{\alpha^2}{4}-1\Big)W^2\Big)dx
 +\sinh(\alpha\pi)\Big(w_x(\pi,t)^2+\Big(\frac{\alpha^2}{2}-1\Big)w(\pi,t)^2\Big)
\]
(Proposition~\ref{prop:flux}), where the last term is the flux through the cut, at which the weight is discontinuous as a function on $\T$. The observation \eqref{eq:pointobs} annihilates this flux, and the continuity of $H^1$ functions, together with $u(x_0,t)=C$, replaces the fixed compact support in the compactness step.

\subsection{Damping, and comparison with linear control}

We first state the nonlinear consequence which uses Theorem \ref{thm:main} directly. Fix a real, nonzero $a\in C^\infty(\T)$, and let
\begin{equation}\label{eq:nldamp}
\begin{cases}
& (1-\partial_x^2)u_t+\partial_x(u+u^2/2)+D_ju=0,\\
 &D_0u=a^2u,\\
 &D_1u=-\partial_x(a^2u_x).
 \end{cases}
\end{equation}

\begin{theorem}\label{thm:nldamp}
For $j=0,1$ and every real $u_0\in H^1(\T)$, \eqref{eq:nldamp} has a unique global solution in $C^1([0,\infty);H^1)$. As $t\to\infty$,
\[
 \begin{gathered}
 u(t)\rightharpoonup m_j\text{ in }H^1,\qquad
 \|u(t)-m_j\|_{H^s}\to0\quad(s<1),\\
 m_0=0,\qquad m_1=\frac1{2\pi}\int_\T u_0\,dx.
 \end{gathered}
\]
\end{theorem}

Rosier \cite{Rosier2016} obtained these weak limits conditionally: for $j=0$ assuming his conjecture, and for $j=1$ assuming its variant with $v=C$ on the strip. Theorem \ref{thm:main} supplies both hypotheses, for data of any size. In \cite{Rosier2016} the damping coefficient is a smooth nonnegative function which is positive on a nonempty open set, rather than a square $a^2$; such a function need not be the square of a smooth function. The proof of Theorem~\ref{thm:nldamp} uses $a^2$ only through its smoothness and nonnegativity and the fact that it does not vanish identically, so the theorem holds verbatim with $a^2$ replaced by any smooth $b\geq0$, $b\not\equiv0$, and it covers the setting of \cite{Rosier2016}. For a nonnegative flux and $j=0$, the same compactness argument is carried out in \cite[Corollary~4.2]{RosierZhang}. The conclusion is weak attraction in the energy space; the proof supplies neither strong $H^1$ attraction nor a decay rate. The Liouville theorem on the line (Proposition~\ref{prop:liouville}) also makes Rosier's boundary stabilization result \cite[Theorem~4]{Rosier2016} unconditional (Corollary~\ref{cor:boundary}).

The rest of this subsection is included for comparison. Except where stated, the results are known or follow from standard arguments, and they do not depend on Theorem \ref{thm:main}. Write $X=H^1(\T)$ and let $S(t)$ be the group for
\begin{equation}\label{eq:linbbm}
 (1-\partial_x^2)u_t+u_x=0.
\end{equation}
All nonlinear states and controls below are real, and the linear equation may be complexified.

\begin{proposition}\label{thm:lin}
If $f\in X$ and $S(t)f$ vanishes on $\omega\times(0,T)$, with $\omega$ nonempty and open, then $f=0$. For the period $2\pi$ and normalization in \eqref{eq:linbbm}, $S(t)f(x_0)=0$ for all $0<t<T$ at a single fixed point $x_0$ also implies $f=0$.
\end{proposition}

Strip continuation for \eqref{eq:linbbm} is equivalent to approximate controllability with a fixed actuator. This is known, see, for instance, \cite[Section~1]{RosierZhang}, and \cite{Micu,ZhangZuazua} for related settings. Section \ref{sec:linear} gives a short spectral proof, using a meromorphic generating series of time derivatives. We have not found the single-point statement in the literature. It relies on the simplicity of the spectrum, which holds for this period and normalization.

For the control statements, forcing is measured in $L^2(0,T;H^{-1}(\T))$, the space naturally dual to the energy space. Define $a_t(x)=a(x+ct)$ for a nonzero speed $c$. The covering condition is
\begin{equation}\label{eq:covering}
 \Theta_T(x):=\int_0^T |a(x+ct)|^2\,dt,
 \qquad \kappa_T:=\min_{x\in\T}\Theta_T(x)>0.
\end{equation}
It is equivalent to requiring that for every $x\in\T$ there is $t\in(0,T)$ with $a(x+ct)\ne0$. Continuity and compactness make this pointwise condition uniform. The physical actuator translates at velocity $-c$. A sufficient condition is $T\geq2\pi/|c|$, since a full circuit gives $$\Theta_{2\pi/|c|}=|c|^{-1}\int_\T a^2\,dx>0.$$ When $\{a\ne0\}$ is large, \eqref{eq:covering} also holds for shorter times. For instance, if $\{a\ne0\}$ is an open arc $I$ of length $|I|<2\pi$, then \eqref{eq:covering} holds if and only if $|c|T>2\pi-|I|$, because an orbit arc of length $|c|T$ avoids $I$ exactly when it fits into the closed complementary arc. This threshold tends to $2\pi/|c|$ as $|I|\to0$, in agreement with \cite[Remark~5.2]{RosierZhang}, where finite speed of propagation for a model KdV--BBM equation in the moving frame shows that the condition $T>2\pi/|c|$ of \cite[Theorem~5.1]{RosierZhang} is sharp when the observation set may be arbitrarily short. Requiring only that each orbit $x+ct$, $t\in[0,T]$, meet the closed support of $a$ is weaker: an orbit may meet that support only at an endpoint time $t\in\{0,T\}$, necessarily at a boundary point of the support, and then $\Theta_T(x)=0$.

\begin{proposition}\label{thm:lincontrol}
For each $T>0$, the equation $$(1-\partial_x^2)u_t+u_x=ah$$ is approximately controllable in $X$ with $h\in L^2(0,T;H^{-1})$. It is exactly controllable in $X$ at some, equivalently every, $T>0$ if and only if $a$ has no zero on $\T$. If $a$ has a zero, it is not null controllable in $X$ at any finite time either. Under \eqref{eq:covering}, there is $C_T>0$ such that
\begin{equation}\label{eq:obs}
 \|f\|_X^2\leq C_T\int_0^T\|a_tS(t)f\|_X^2\,dt\qquad(f\in X).
\end{equation}
Consequently, $$(1-\partial_x^2)u_t+u_x=a_th$$ is exactly controllable in $X$ with controls in $L^2(0,T;H^{-1})$.
\end{proposition}

The failure of exact controllability for a fixed actuator supported in a subinterval, caused by the accumulation of the spectrum at $0$, is known; see \cite[Section~1]{RosierZhang} and \cite{Micu}. It motivated the moving controls of \cite{RosierZhang}. Proposition \ref{thm:lincontrol} states it in sharp form: a single zero of $a$ suffices. For $T>2\pi/|c|$, the moving exact controllability is the case $s=1$ of \cite[Proposition~6.2]{RosierZhang}, written in the original frame.

The two natural linear feedback conclusions have different strengths.

\begin{proposition}\label{thm:linstab}
For every nonzero real $a\in C^\infty(\T)$, the solutions of
\begin{equation*}
 (1-\partial_x^2)u_t+u_x+a^2u=0
\end{equation*}
converge strongly to zero in $X$. Its solution semigroup has operator norm $1$ at every finite time, so no uniform decay estimate tending to zero holds on the unit ball. For $c\ne0$, the moving feedback equation
\begin{equation}\label{eq:linmovingdamp}
 (1-\partial_x^2)u_t+u_x=-a_t(1-\partial_x^2)(a_tu)
\end{equation}
is uniformly exponentially stable: its evolution family satisfies
\[
 \|U(t,s)\|_{X\to X}\leq M e^{-\gamma(t-s)}\qquad(t\geq s\geq0),
\]
with $M,\gamma>0$ depending only on $a$ and $c$.
\end{proposition}

In the moving frame $v(x,t)=u(x-ct,t)$, which is an isometry of $X$, the second assertion is the case $s=1$ of \cite[Lemma~6.3]{RosierZhang}. We include a short proof in the original frame.

Finally, the bounded right inverse furnished by moving linear observability gives the following local nonlinear result.

\begin{proposition}\label{thm:nlcontrol}
Assume \eqref{eq:covering}. There are $\varepsilon,K>0$, depending on $a,c,T$, such that every pair $u_0,u_1\in H^1(\T;\R)$ satisfying
\[
 \|u_0\|_{H^1}+\|u_1\|_{H^1}<\varepsilon
\]
can be joined by a real control $h\in L^2(0,T;H^{-1})$ for the equation
\begin{equation}\label{eq:nlmoving}
 (1-\partial_x^2)u_t+\partial_x(u+u^2/2)=a_th
\end{equation}
from $u(0)=u_0$ to $u(T)=u_1$, with $u\in C([0,T];H^1)$ and
\[
 \|h\|_{L^2(0,T;H^{-1})}+\|u\|_{C([0,T];H^1)}
 \leq K(\|u_0\|_{H^1}+\|u_1\|_{H^1}).
\]
\end{proposition}

For $T>2\pi/|c|$, Proposition \ref{thm:nlcontrol} is the case $s=1$ of \cite[Theorem~6.1]{RosierZhang}, written in the original frame. Rosier and Zhang also prove global exact controllability in large time and semiglobal exponential stabilization \cite[Theorem~1.1 and Section~6]{RosierZhang}. Section \ref{sec:control} proves Propositions \ref{thm:lincontrol}--\ref{thm:nlcontrol}. The direct applications of our nonlinear continuation results are Theorem~\ref{thm:nldamp} and Corollary~\ref{cor:boundary}, proved in Section~\ref{sec:nonlineardamping}.

\subsection{Organization of the paper}
Section~\ref{sec:weighted} contains the nonlinear continuation results. It
proves Theorem~\ref{thm:main} through a Liouville theorem on the line
(Proposition~\ref{prop:liouville}) and Theorem~\ref{thm:point} through the
weighted identity on a fundamental domain, and it closes with general
polynomial fluxes and the level $C=-1$. Section~\ref{sec:linear} proves the
linear continuation result (Proposition~\ref{thm:lin}).
Section~\ref{sec:control} proves the linear control and stabilization
statements and local nonlinear moving control.
Section~\ref{sec:nonlineardamping} proves Theorem~\ref{thm:nldamp} and
Corollary~\ref{cor:boundary}. Section~\ref{sec:remarks} contains further
remarks and open questions.

\section*{Declaration of generative AI}Large language models have been of great importance in the production of this
manuscript. Below is a brief narration of how they contributed to the current work.

During the preparation of this work, the authors used generative-AI assistants,
including Claude Opus 5.5 (Anthropic), for literature searches, for mathematical
exploration (in particular, the fundamental-domain version of the weighted
identity with the flux term, Proposition~\ref{prop:flux}, and its consequence,
Theorem~\ref{thm:point}), and for writing and running the numerical checks
described in Remark~\ref{rem:numerics}. After using these tools, the authors
reviewed and edited the content as needed, verified all mathematical arguments,
and take full responsibility for the content of the published article.

\section{A weighted Hamiltonian and nonlinear unique continuation}\label{sec:weighted}

\subsection{The flow and analyticity in time}\label{sec:flow}

For $\beta\in\R$, consider on $\T$ or $\R$ the equation
\begin{equation}\label{eq:beta}
\begin{cases}
 &(1-\partial_x^2)z_t+\partial_xF_\beta(z)=0,\\
  &F_\beta(z)=\beta z+\frac{z^2}{2}.
  \end{cases}
\end{equation}
On either domain the inverse $R=(1-\partial_x^2)^{-1}$ is the Fourier multiplier $(1+\xi^2)^{-1}$, with integer frequencies on $\T$. Since $H^1$ is an algebra, $-R\partial_xF_\beta$ is a polynomial, locally Lipschitz map from $H^1$ to $H^2$, hence to $H^1$. The Banach-space ODE gives a local $C^1$ flow. For real solutions,
\begin{equation}\label{eq:conservation}
 \frac12\frac{d}{dt}\|z\|_{H^1}^2
 =\int F_\beta(z)z_x\,dx
 =\int\partial_x\left(\frac{\beta z^2}{2}+\frac{z^3}{6}\right)dx=0.
\end{equation}
Density justifies the calculation in $H^1$; on the line the last integral also vanishes by approximation. The conserved bound prevents finite-time blow-up in either time direction. A function $z\in C(I;H^1)$ solving \eqref{eq:beta} in the sense of distributions is a trajectory of this flow: since $1-\partial_x^2$ is an isomorphism from $H^1$ onto $H^{-1}$, the equation gives $$z_t=-R\partial_xF_\beta(z)\in C(I;H^1).$$

The same polynomial extends holomorphically to the complexified $H^1$ space. The local Picard contraction on a sufficiently small complex time disk produces a holomorphic solution, agreeing with the real flow by uniqueness. Applying this at each real time proves that the real trajectory is real analytic on $\R$. A uniform complex-time radius is not needed. For $\beta=1$ this is \cite[Proposition~2.2]{RosierZhang}.

We shall also use the following standard facts on $\T$. Every $f\in H^1(\T)$ has a continuous representative, with $$|f(x)-f(y)|\leq|x-y|^{1/2}\|f_x\|_2,$$ hence bounded subsets of $H^1(\T)$ are relatively compact in $C(\T)$ by the Arzel\`a--Ascoli theorem, and point evaluations are continuous linear functionals on $H^1(\T)$. Moreover $H^3(\T)\subset C^2(\T)$, and on $\T$ the resolvent is the convolution $Rf=G*f$ with the positive kernel $G$ introduced after \eqref{eq:defpsi}. Since $z\mapsto RF_\beta(z)$ is a continuous polynomial map from $H^1(\T)$ to $H^3(\T)$, the analyticity of $t\mapsto z(t)\in H^1(\T)$ implies that, for every $x_0\in\T$, the functions $t\mapsto z(x_0,t)$, $t\mapsto(RF_\beta(z(t)))(x_0)$ and $t\mapsto\partial_x(RF_\beta(z(t)))(x_0)$ are real analytic on $\R$.

\subsection{Shift of the level}

For $C\in\R$ we set
\begin{equation*}
 \beta=1+C,\\ \quad q_C=C+\frac{C^2}{2},\\ \quad\text{and } \mathcal H_\beta(z)=\frac{\beta z^2}{2}+\frac{z^3}{6},
\end{equation*}
so that $\mathcal H_\beta'=F_\beta$ and $\mathcal H_\beta(0)=0$.

\begin{lemma}\label{lem:shift}
Let $u\in C(\R;H^1(\T;\R))$ solve \eqref{eq:bbm}, let $\psi$ be given by \eqref{eq:defpsi}, and let $C\in\R$. Then $z=u-C$ is a solution of \eqref{eq:beta} in $C^1(\R;H^1(\T))$, the function $w:=RF_\beta(z)$ satisfies $$w=\psi-q_C\in C(\R;H^3(\T)),$$ and, for every $t\in\R$,
\begin{equation}\label{eq:zw}
\begin{cases}
  &z_t=-w_x,\\
  &w-w_{xx}=F_\beta(z),\\ 
  &\|z(t)\|_{H^1(\T)}=\|z(0)\|_{H^1(\T)}.
  \end{cases}
\end{equation}
\end{lemma}

\begin{proof}
Since $$u+\frac{u^2}{2}=\beta z+\frac{z^2}{2}+C+\frac{C^2}{2},$$ we have $q=q_C+F_\beta(z)$, and substitution into \eqref{eq:bbm} shows that $z$ solves \eqref{eq:beta}. By Section~\ref{sec:flow}, $z\in C^1(\R;H^1(\T))$ is a trajectory of the flow. As $R$ maps constants to themselves, $\psi=Rq=q_C+RF_\beta(z)$, that is, $w=\psi-q_C$; then $z_t=u_t=-\psi_x=-w_x$ and $w-w_{xx}=F_\beta(z)$. The last identity in \eqref{eq:zw} is \eqref{eq:conservation}.
\end{proof}

\subsection{A Liouville theorem on the line, and the proof of Theorem~\ref{thm:main}}

\begin{lemma}\label{lem:cut}
If a real $H^1(\T)$ solution of \eqref{eq:beta} vanishes on $\omega\times(0,T)$, it determines a solution $\widetilde z\in C^1(\R;H^1(\R))$ of \eqref{eq:beta} and a fixed compact interval $[A,B]$ such that
$\supp\widetilde z(t)\subset[A,B]$ for all $t\in\R$. Its $H^1(\R)$ norm is uniformly bounded, and $\widetilde z=0$ implies $z=0$.
\end{lemma}

\begin{proof}
Restriction to $\omega$ is a bounded linear map from $H^1(\T)$ to $L^2(\omega)$. Since $t\mapsto z(t)$ is real analytic (see Section~\ref{sec:flow}), the map $t\mapsto z(t)|_\omega\in L^2(\omega)$ is real analytic and vanishes on $(0,T)$. By the identity theorem it vanishes for all $t\in\R$. Choose a cut point strictly inside an interval contained in $\omega$. In the corresponding period, $z(t)$ and $z_t(t)$ vanish on fixed neighborhoods (collars) of both endpoints for all $t$. Extend that period by zero. Equivalently, multiply the periodic lift by a fixed smooth cutoff which is one outside the zero collars and vanishes near the endpoints, and then extend by zero. This is a bounded linear map into $H^1(\R)$ and gives a $C^1$ trajectory $\widetilde z$ supported in a fixed interval $[A,B]$. There are no interface contributions: $z$, $z_t$ and $F_\beta(z)$ vanish on the collars because $F_\beta(0)=0$, and the operators in \eqref{eq:beta} are local, hence $\widetilde z$ solves \eqref{eq:beta} in the sense of distributions on the line. Its $H^1(\R)$ norm equals the $H^1(\T)$ norm of $z$, which is conserved by \eqref{eq:conservation}. The propagation to all times is as in the proof of \cite[Theorem~3.1]{RosierZhang}, and the extension by zero is the one used in the proof of \cite[Theorem~4.1]{RosierZhang}.
\end{proof}

\begin{lemma}\label{lem:resolventsupport}
Let $z\in C^1(\R;H^1(\R))$ solve \eqref{eq:beta} with support in a fixed interval $[A,B]$ for all times. Then
\begin{equation}\label{eq:w}
\begin{cases}
& w(t)=RF_\beta(z(t))\in H^3(\R),\\
 &z_t=-w_x,\\ 
 &\supp w(t)\subset[A,B].
 \end{cases}
\end{equation}
\end{lemma}

\begin{proof}
The algebra property and the resolvent multiplier give $w\in C(\R;H^3(\R))$. As $1-\partial_x^2$ is an isomorphism from $H^1(\R)$ onto $H^{-1}(\R)$, the equation on the line gives $$z_t=-R\partial_xF_\beta(z)=-w_x.$$ Fixed support implies that $z_t$ vanishes on both exterior half-lines. Thus $w$ is constant on each half-line, and $w\in L^2(\R)$ forces both constants to be zero. Notice that compact support of $F_\beta(z)$ alone would not imply compact support of its resolvent. Given $\supp F_\beta(z(t))\subset[A,B]$, pairing $F_\beta(z)=w-w_{xx}$ with $e^{\pm x}$, or using the kernel $e^{-|x|}/2$ of $R$ on the line, shows that the support assertion for $w(t)$ is equivalent to $$\int_\R F_\beta(z)e^{\pm x}\,dx=0,$$ for a nonnegative flux, the identity with $e^x$ is the argument of \cite[Theorem~4.1]{RosierZhang}.
\end{proof}

\begin{lemma}\label{lem:weighted}
For a trajectory as in Lemma \ref{lem:resolventsupport} and $\alpha\in\R$, define
\[
 N_\alpha(t)=\int_\R e^{\alpha x}
 \left(\frac{\beta z(x,t)^2}{2}+\frac{z(x,t)^3}{6}\right)dx,
 \qquad\text{and}\quad W(x,t)=e^{\alpha x/2}w(x,t).
\]
Then $N_\alpha$ is continuously differentiable and
\begin{equation}\label{eq:weighted}
 N_\alpha'(t)=-\frac\alpha2\left(\|W_x(t)\|_2^2+
 \left(\frac{\alpha^2}{4}-1\right)\|W(t)\|_2^2\right).
\end{equation}
If $\alpha>2$ and $z$ is uniformly bounded in $H^1$, then $N_\alpha$ is bounded and nonincreasing on $\R$, and
$$\int_\R\|w(t)\|_2^2\,dt<\infty.$$
\end{lemma}

\begin{proof}
The functional defining $N_\alpha$ is $C^1$ on the closed subspace of $H^1(\R)$ consisting of functions supported in $[A,B]$. Indeed, on this subspace the exponential weight may be replaced by a bounded smooth function agreeing with it near $[A,B]$. This also justifies the chain rule, with $$N_\alpha'(t)=\int_\R e^{\alpha x}F_\beta(z)z_t\,dx,$$ no assertion about an exponential functional on all of $H^1(\R)$ is needed. By \eqref{eq:w},
\begin{align*}
 N_\alpha'(t)&=-\int_\R e^{\alpha x}(w-w_{xx})w_x\,dx=\frac\alpha2\int_\R e^{\alpha x}(w^2-w_x^2)\,dx.
\end{align*}
All boundary terms vanish because $w\in H^3$ has compact support. Furthermore, $$e^{\alpha x/2}w_x=W_x-\alpha W/2$$ and the cross term $$-\alpha\int WW_x=-\frac\alpha2\int(W^2)_x$$ vanishes, so that
\[
 \int_\R e^{\alpha x}w_x^2\,dx
 =\int_\R\left(W_x-\frac\alpha2W\right)^2dx
 =\|W_x\|_2^2+\frac{\alpha^2}{4}\|W\|_2^2,
\]
which proves \eqref{eq:weighted}. The fixed support and the uniform $H^1$ and $L^\infty$ bounds make $|N_\alpha|$ uniformly bounded. If $\alpha>2$, integration of \eqref{eq:weighted} on $[-r,r]$, followed by $r\to\infty$, gives
$$\frac\alpha2\big(\frac{\alpha^2}{4}-1\big)\int_\R\|W(t)\|_2^2\,dt\leq2\sup_\R|N_\alpha|,$$ on $[A,B]$ the weighted and unweighted $L^2$ norms are equivalent.
\end{proof}

\begin{lemma}\label{lem:compactzero}
Let $\alpha\in\R$, and suppose $z_n$ are real, uniformly bounded in $H^1(\R)$, and supported in $[A,B]$. If $RF_\beta(z_n)\to0$ in $L^2(\R)$, then $z_n\to0$ in $L^2(\R)$ and
\[
 \int_\R e^{\alpha x}\left(\frac{\beta z_n^2}{2}+\frac{z_n^3}{6}\right)dx\longrightarrow0.
\]
\end{lemma}

\begin{proof}
The set $\{v\in H^1(\R):\supp v\subset[A,B]\}$ embeds compactly into $L^2(\R)$. Hence every subsequence has a further subsequence converging strongly in $L^2$ and weakly in $H^1$ to a function $z_*\in H^1(\R)$ supported in $[A,B]$. The uniform $L^\infty$ bound gives $F_\beta(z_n)\to F_\beta(z_*)$ in $L^2$. Since $R$ is bounded on $L^2$, $RF_\beta(z_n)\to RF_\beta(z_*)$, so $RF_\beta(z_*)=0$, and $F_\beta(z_*)=0$ because $R$ is injective. The continuous representative of $z_*$ therefore takes values in the discrete set $\{0,-2\beta\}$ (a single point if $\beta=0$), and it vanishes outside $[A,B]$. Connectedness of $\R$ forces $z_*=0$. Every subsequential limit is consequently zero. The uniform $L^\infty$ bound, the fixed support and strong $L^2$ convergence make both the quadratic and cubic integrals tend to zero.
\end{proof}

\begin{proposition}\label{prop:liouville}
Let $z\in C(\R;H^1(\R))$ be a real solution of \eqref{eq:beta} which is supported in a fixed compact interval $[A,B]$ for all $t$ in a nonempty open time interval. Then $z\equiv0$.
\end{proposition}

\begin{proof}
By Section~\ref{sec:flow}, $z\in C^1(\R;H^1(\R))$ is a trajectory of the flow, its $H^1(\R)$ norm is conserved, and $t\mapsto z(t)$ is real analytic. Restriction to $\R\setminus[A,B]$ is a bounded linear map into $L^2(\R\setminus[A,B])$; by the identity theorem, $z(t)$ is supported in $[A,B]$ for all $t\in\R$. Fix $\alpha>2$. Lemma \ref{lem:weighted} supplies sequences $t_n^+\to+\infty$ and $t_n^-\to-\infty$ along which $\|w(t_n^\pm)\|_2\to0$. By Lemma \ref{lem:compactzero}, $N_\alpha(t_n^\pm)\to0$. For each fixed $t$, monotonicity gives, for large $n$,
\[
 N_\alpha(t_n^-)\geq N_\alpha(t)\geq N_\alpha(t_n^+).
\]
Hence $N_\alpha(t)=0$ for every $t$, and $N_\alpha'\equiv0$. Since $\alpha>2$, \eqref{eq:weighted} implies $W=0$, hence $w=0$, and then $F_\beta(z)=w-w_{xx}=0$. The continuity and connectedness argument of Lemma \ref{lem:compactzero} gives $z=0$. Both time directions are essential to this argument: a single zero limit would not pin a monotone functional of indefinite sign to zero.
\end{proof}

\begin{proof}[Proof of Theorem~\ref{thm:main}]
By Lemma~\ref{lem:shift}, $z=u-C$ is a real $H^1(\T)$ solution of \eqref{eq:beta} with $\beta=1+C$, and it vanishes on $\omega\times(0,T)$. By Lemma~\ref{lem:cut}, it determines a solution $\widetilde z$ of \eqref{eq:beta} on the line, supported in a fixed compact interval for all times. Proposition~\ref{prop:liouville} gives $\widetilde z=0$, hence $z=0$, and returning to the fundamental period shows $u=C$ on $\T\times\R$. For $C=0$ this is Rosier's conjecture, with no restriction on $u_0\in H^1(\T)$.
\end{proof}

\subsection{Observation at a single point}

\begin{lemma}\label{lem:strip}
Assume that $u=C$ on $\omega\times(0,T)$, where $\omega\subset\T$ is nonempty and open. Then, for every $t\in(0,T)$, the functions $z(t)$, $z_t(t)$, $w(t)$ and $w_x(t)$ of Lemma~\ref{lem:shift} vanish identically on $\omega$. In particular, \eqref{eq:pointobs} holds at every point $x_0\in\omega$.
\end{lemma}

\begin{proof}
Fix $t\in(0,T)$. By assumption $z(t)=0$ on $\omega$. For $|s|$ small, the difference quotients $(z(t+s)-z(t))/s$ vanish on $\omega$ and converge to $z_t(t)$ in $H^1(\T)$, hence uniformly on $\T$, thus $z_t(t)=0$ on $\omega$, and $w_x(t)=-z_t(t)=0$ on $\omega$. Let $J$ be a connected component of $\omega$ (an open arc, or $\T$ itself). Since $w(t)\in H^3(\T)\subset C^2(\T)$ and $w_x(t)=0$ on $J$, $w(t)$ is constant on $J$ and $w_{xx}(t)=0$ on $J$. The identity $w=F_\beta(z)+w_{xx}$ between continuous functions then gives $w(t)=F_\beta(0)=0$ on $J$. Finally, $u(x_0,t)=C$ and $\psi(x_0,t)=q_C$ mean $z(x_0,t)=0$ and $w(x_0,t)=0$.
\end{proof}

\begin{lemma}\label{lem:allt}
Assume \eqref{eq:pointobs}. Then $z(x_0,t)=w(x_0,t)=w_x(x_0,t)=0$ for every $t\in\R$.
\end{lemma}

\begin{proof}
By Section~\ref{sec:flow}, the functions $t\mapsto z(x_0,t)$ and $t\mapsto w(x_0,t)$ are real analytic on $\R$, and by \eqref{eq:pointobs} they vanish on $(0,T)$. By the identity theorem they vanish on $\R$. Since the evaluation at $x_0$ is a continuous linear functional on $H^1(\T)$ and $z\in C^1(\R;H^1(\T))$, we have $\frac{d}{dt}z(x_0,t)=z_t(x_0,t)$; hence $w_x(x_0,t)=-z_t(x_0,t)=0$ by \eqref{eq:zw}.
\end{proof}

We identify functions on $\T$ with $2\pi$-periodic functions on $\R$ and restrict them to the fundamental domain $[-\pi,\pi]$. On this interval the weight $e^{\alpha x}$ is smooth, but it does not extend to a continuous function on $\T$, this produces a flux term at the cut $x=\pm\pi$.

\begin{proposition}\label{prop:flux}
Let $u$, $C$, $z$ and $w$ be as in Lemma~\ref{lem:shift}. For $\alpha\in\R$ define
\[
 \widetilde N_\alpha(t)=\int_{-\pi}^{\pi}e^{\alpha x}\,\mathcal H_\beta\big(z(x,t)\big)\,dx,\qquad
 W(x,t)=e^{\alpha x/2}w(x,t)\quad(x\in[-\pi,\pi]).
\]
Then $\widetilde N_\alpha\in C^1(\R)$ and, for every $t\in\R$,
\begin{equation}\label{eq:fluxidentity}
 \widetilde N_\alpha'(t)=-\frac\alpha2\int_{-\pi}^{\pi}\Big(W_x^2+\Big(\frac{\alpha^2}{4}-1\Big)W^2\Big)dx
 +\sinh(\alpha\pi)\Big(w_x(\pi,t)^2+\Big(\frac{\alpha^2}{2}-1\Big)w(\pi,t)^2\Big).
\end{equation}
\end{proposition}

\begin{proof} We will give the proof in three steps. 

\vspace{0.2cm}
\noindent \textbf{Step 1.} \emph{Time derivative.} 
\vspace{0.2cm}

For $f,g\in H^1(\T)$, $$\mathcal H_\beta(f+g)-\mathcal H_\beta(f)-F_\beta(f)g=\big(\frac\beta2+\frac f2\big)g^2+\frac{g^3}6,$$ so the functional $$f\mapsto\int_{-\pi}^{\pi}e^{\alpha x}\mathcal H_\beta(f)\,dx$$ is Fr\'echet differentiable on $H^1(\T)$, with derivative $$g\mapsto\int_{-\pi}^{\pi}e^{\alpha x}F_\beta(f)g\,dx$$ depending continuously on $f$. By the chain rule and \eqref{eq:zw}, $\widetilde N_\alpha\in C^1(\R)$ and
\[
 \widetilde N_\alpha'(t)=\int_{-\pi}^{\pi}e^{\alpha x}F_\beta(z)z_t\,dx
 =-\int_{-\pi}^{\pi}e^{\alpha x}ww_x\,dx+\int_{-\pi}^{\pi}e^{\alpha x}w_{xx}w_x\,dx .
\]
In the rest of the proof $t$ is fixed and omitted, and all integrals are over $[-\pi,\pi]$.

\vspace{0.2cm}
\noindent \textbf{Step 2.} \emph{Integration by parts.} 
\vspace{0.2cm}

The functions $w,w_x\in H^2(\T)\subset C^1(\T)$ are $2\pi$-periodic, so
$$\big[e^{\alpha x}f(x)\big]_{-\pi}^{\pi}=(e^{\alpha\pi}-e^{-\alpha\pi})f(\pi)=2\sinh(\alpha\pi)f(\pi),$$
for $f\in\{w^2,w_x^2\}$. Since $2ww_x=(w^2)_x$ and $2w_xw_{xx}=(w_x^2)_x$,
\begin{gather*}
 -\int e^{\alpha x}ww_x=-\sinh(\alpha\pi)\,w(\pi)^2+\frac\alpha2\int e^{\alpha x}w^2
 \end{gather*}
 and
 \begin{gather*}
 \int e^{\alpha x}w_{xx}w_x=\sinh(\alpha\pi)\,w_x(\pi)^2-\frac\alpha2\int e^{\alpha x}w_x^2 .
\end{gather*}
Hence
\begin{equation}\label{eq:flux0}
 \widetilde N_\alpha'=\frac\alpha2\int e^{\alpha x}\big(w^2-w_x^2\big)+\sinh(\alpha\pi)\big(w_x(\pi)^2-w(\pi)^2\big).
\end{equation}

\vspace{0.2cm}
\noindent \textbf{Step 3.}\emph{The conjugated function.} 
\vspace{0.2cm}

We have $e^{\alpha x}w^2=W^2$ and $e^{\alpha x/2}w_x=W_x-\frac\alpha2W$, so that
\begin{gather*}
 \int e^{\alpha x}w_x^2=\int W_x^2-\alpha\int WW_x+\frac{\alpha^2}4\int W^2
 \end{gather*}
 and
 \begin{gather*}
 \int WW_x=\frac12\big[W^2\big]_{-\pi}^{\pi}=\frac12\big(e^{\alpha\pi}-e^{-\alpha\pi}\big)w(\pi)^2=\sinh(\alpha\pi)\,w(\pi)^2,
\end{gather*}
by the periodicity of $w$. Inserting this in \eqref{eq:flux0},
\[
 \widetilde N_\alpha'=\frac\alpha2\int W^2-\frac\alpha2\int W_x^2+\frac{\alpha^2}2\sinh(\alpha\pi)\,w(\pi)^2
 -\frac{\alpha^3}8\int W^2+\sinh(\alpha\pi)\big(w_x(\pi)^2-w(\pi)^2\big),
\]
which is \eqref{eq:fluxidentity}.
\end{proof}

\begin{remark}
Let us give some remarks.
\begin{itemize}
\item[(i)] For $\alpha>2$, the first term in \eqref{eq:fluxidentity} is nonpositive and vanishes only if $w(t)=0$. Hence $\widetilde N_\alpha$ is nonincreasing as soon as the flux term vanishes, although its density $\mathcal H_\beta$ has no sign. Since the flux term is nonnegative for $\alpha\geq\sqrt2$, without an observation at the cut it works against this monotonicity.
\item[(ii)]  A smooth periodic weight cannot replace $e^{\alpha x}$ in this argument: for such a weight $\varphi$ the same computation gives $$\frac{d}{dt}\int_\T\varphi\,\mathcal H_\beta(z)\,dx=\frac12\int_\T\varphi_x\big(w^2-w_x^2\big)\,dx,$$ with no flux term, and since $$\int_\T\varphi_x\,dx=0$$ this quadratic form is indefinite unless $\varphi$ is constant (in which case one recovers the conservation of the Hamiltonian). The price for the non-periodic weight $e^{\alpha x}$ is the flux term, which involves only the two numbers $w(\pi,t)$ and $w_x(\pi,t)$.
\item[(iii)]  If $z$ vanishes near $\pm\pi$ for all $t$ in an open interval, then so do $w(t)$ and $w_x(t)$ for $t$ in that interval (Lemma~\ref{lem:strip}), and \eqref{eq:fluxidentity} reduces to the identity \eqref{eq:weighted} for the extension of $z$ by zero to the line.
\end{itemize}
\end{remark}

\begin{proof}[Proof of Theorem~\ref{thm:point}]
The equation commutes with translations, and $\psi$ is translated with $u$, hence we may assume that $x_0=\pi$ (the point $\pi\equiv-\pi$ of $\T$). Let $z$ and $w$ be as in Lemma~\ref{lem:shift}, fix $\alpha=3$, and write $\widetilde N=\widetilde N_3$. By Lemma~\ref{lem:allt}, for every $t\in\R$,
\begin{equation}\label{eq:cutvanish}
 z(\pi,t)=0,\qquad w(\pi,t)=0,\qquad w_x(\pi,t)=0 .
\end{equation}

\smallskip
\noindent\emph{\textbf{Step 1: monotonicity and integrability.}} 
\smallskip

By Proposition~\ref{prop:flux} and \eqref{eq:cutvanish},
\begin{equation}\label{eq:mono3}
 \widetilde N'(t)=-\frac32\int_{-\pi}^{\pi}\Big(W_x^2+\frac54W^2\Big)dx\leq0\qquad\text{for all }t\in\R .
\end{equation}
Let $\nu_1=\|z(0)\|_{H^1(\T)}$, so that $\|z(t)\|_{H^1}=\nu_1$ for all $t$ by \eqref{eq:zw}, and let $\nu_\infty=\sup_t\|z(t)\|_\infty$, which is finite by the embedding $H^1(\T)\subset C(\T)$. Since $$|\mathcal H_\beta(z)|\leq\left(\frac{|\beta|}2+\frac{\nu_\infty}6\right)z^2,$$ we get $$|\widetilde N(t)|\leq\widetilde N_{\max}:=e^{3\pi}\left(\frac{|\beta|}2+\frac{\nu_\infty}6\right)\nu_1^2,$$ for all $t$. Integrating \eqref{eq:mono3} over $[-r,r]$ and letting $r\to\infty$ gives $$\frac{15}8\int_\R\|W(t)\|_{L^2(-\pi,\pi)}^2\,dt\leq2\widetilde N_{\max}.$$ Since $w^2=e^{-3x}W^2\leq e^{3\pi}W^2$ on $[-\pi,\pi]$, it follows that $$\int_\R\|w(t)\|_{L^2(\T)}^2\,dt<\infty.$$ As $t\mapsto\|w(t)\|_2$ is continuous, there are sequences $t_n^+\to+\infty$ and $t_n^-\to-\infty$ with $\|w(t_n^\pm)\|_2\to0$.

\smallskip
\noindent\emph{\textbf{Step 2: compactness at both ends of time.}}
\smallskip

Let $(s_n)$ be a sequence of times with $\|w(s_n)\|_2\to0$. We claim that $z(s_n)\to0$ uniformly on $\T$, hence $\widetilde N(s_n)\to0$. By Section~\ref{sec:flow}, every subsequence of $(z(s_n))$ has a further subsequence converging uniformly to some $z_*\in C(\T)$. Along it, $$F_\beta(z(s_n))\to F_\beta(z_*)$$ uniformly, hence in $L^2(\T)$, and $$w(s_n)=RF_\beta(z(s_n))\to RF_\beta(z_*)$$ in $L^2(\T)$, because $R$ is bounded on $L^2$. Thus $RF_\beta(z_*)=0$, and $F_\beta(z_*)=0$ because $R$ is injective. In other words, the continuous function $z_*$ takes values in the set $\{0,-2\beta\}$, which has at most two points. Since $\T$ is connected, $z_*$ is constant, and $$z_*(\pi)=\lim z(s_n)(\pi)=0,$$ by \eqref{eq:cutvanish}. Hence $z_*=0$. Every subsequence having a further subsequence converging uniformly to $0$, the whole sequence $z(s_n)$ converges uniformly to $0$, and $$|\widetilde N(s_n)|\leq e^{3\pi}\big(\frac{|\beta|}2+\frac{\|z(s_n)\|_\infty}6\big)\|z(s_n)\|_2^2\to0.$$

\smallskip
\noindent\emph{\textbf{Step 3: conclusion.}}
\smallskip

Fix $t\in\R$. For $n$ large we have $t_n^-<t<t_n^+$, and since $\widetilde N$ is nonincreasing, $$\widetilde N(t_n^-)\geq\widetilde N(t)\geq\widetilde N(t_n^+).$$ By Step~2 both bounds tend to $0$, so that $\widetilde N(t)=0$. Hence $\widetilde N\equiv0$, $\widetilde N'\equiv0$, and \eqref{eq:mono3} gives $W(t)=0$, that is, $w(t)=0$, for every $t$. By \eqref{eq:zw}, $F_\beta(z(t))=w-w_{xx}=0$, and the argument of Step~2 shows that $z(t)$ is a constant belonging to $\{0,-2\beta\}$, with $z(\pi,t)=0$. Therefore $z\equiv0$, that is, $u\equiv C$.
\end{proof}

Combined with Lemma~\ref{lem:strip}, Theorem~\ref{thm:point} gives a second proof of Theorem~\ref{thm:main}.


\begin{remark}\label{rem:general}
Only three properties of the quadratic nonlinearity were used above:
\begin{enumerate}
\item $z\mapsto R\partial_xF_\beta(z)$ is analytic on $H^1(\T)$ (Section~\ref{sec:flow});
\item $F_\beta$ has an antiderivative $\mathcal H_\beta$ with $\mathcal H_\beta(0)=0$, which gives \eqref{eq:conservation} and the weighted identities \eqref{eq:weighted} and \eqref{eq:fluxidentity};
\item $0$ is an isolated zero of $F_\beta$ (the compactness steps).
\end{enumerate}
Let $f$ be a nonconstant real polynomial, and consider $$(1-\partial_x^2)u_t+\partial_xf(u)=0\qquad \text{on}\quad \T.$$ Put $z=u-C$, $F(z)=f(z+C)-f(C)$ and $\psi=Rf(u)$. All three properties then hold, with $$\mathcal H(z)=\int_0^zF,$$ the zero set of $F$ is finite and contains $0$. The proofs above therefore apply verbatim, with the weighted functionals built on $\mathcal H$ and with $w=\psi-f(C)=RF(z)$. So every real $H^1(\T)$ solution that equals a constant $C$ on a nonempty strip $\omega\times(0,T)$ is identically equal to $C$; more generally, $u(x_0,t)=C$ and $\psi(x_0,t)=f(C)$ for $t\in(0,T)$ imply $u\equiv C$. No sign condition on $f$ is needed. This contrasts with \cite[Theorem~4.1]{RosierZhang} and \cite{daSilvaFreire}, which require a one-signed flux but no analyticity. If $f$ is constant, stationary solutions show that the conclusion fails. For $f(u)=u$ one recovers a continuation result for the linearized equation; in that case the observation of $u$ alone at one point is sufficient (Proposition~\ref{thm:lin}).
\end{remark}


\begin{remark}\label{rem:level}
For $C=-1$ we have $\beta=0$ and $F_0(z)=z^2/2\geq0$: this is the case covered by \cite[Theorem~4.1]{RosierZhang}, whose proof uses the positivity of the flux through the exponential moment $$\int F_0(z)e^{x}\,dx.$$ In this case Theorem~\ref{thm:point} is in fact elementary: since $$\psi+\frac12=R\big(\frac12(1+u)^2\big)$$ and $G>0$, the condition $\psi(x_0,t)=-\frac12$ at a single time $t$ forces $u(t)\equiv-1$, hence $u\equiv-1$ by uniqueness. On the other hand, the equation $$(1-\partial_x^2)\tilde u_t+\partial_x(\frac12\tilde u^2)=0$$ satisfied by $\tilde u=1+u$ is invariant under $$\tilde u(x,t)\mapsto-\tilde u(2x_0-x,t),$$ equivalently, \eqref{eq:bbm} is invariant under $$u(x,t)\mapsto-2-u(2x_0-x,t).$$ Hence, if $1+u_0$ is odd with respect to $x_0$ and not identically zero, for instance $$u_0(x)=-1+\sin(x-x_0),$$ and uniqueness shows that $1+u(t)$ is odd with respect to $x_0$ for all $t$, so that $u(x_0,t)=-1$ for all $t$, while $u$ is not constant. Thus the second condition in \eqref{eq:pointobs} cannot be dropped at the level $C=-1$. For $\beta\neq0$, $F_\beta$ changes sign, and the weighted Hamiltonian uses no sign condition at all.
\end{remark}

\section{Linear continuation by a meromorphic generating series}\label{sec:linear}

On the complexification of $X$, define
\begin{equation*}
 L=-R\partial_x,\qquad S(t)=e^{tL},\qquad
 \lambda_k=-\frac{ik}{1+k^2}\quad(k\in\Z).
\end{equation*}
The operator $L$ is bounded, compact, and skew-adjoint on $H^s(\T)$ for every $s\in\R$. Thus $S(t)$ is unitary for real $t$, while $S(\zeta)$ is entire with norm at most $e^{|\operatorname{Im}\zeta|/2}$. The eigenvalues are simple: equality for indices $j,k$ implies $(j-k)(1-jk)=0$, and distinct integers cannot have product $1$. This simplicity depends on the stated period and normalization.

\begin{proof}[Proof of Proposition \ref{thm:lin}]
Fix a point $x_0$, and expand $f$ in its Fourier series $\sum_{k\in\Z}\widehat f(k)e^{ikx}$. Since $f\in H^1$, the coefficients $b_k=\widehat f(k)e^{ikx_0}$ are absolutely summable. Hence the function
$g(\zeta)=S(\zeta)f(x_0)$, which equals $\sum_k b_ke^{\lambda_k\zeta}$, is entire. Vanishing on a real interval implies $g\equiv0$, so
$m_n=g^{(n)}(0)=\sum_k b_k\lambda_k^n=0$ for every $n\geq0$. For $|\eta|<2$, absolute convergence gives the ordinary generating series
\begin{equation}\label{eq:meromorphic}
 \sum_{n=0}^\infty m_n\eta^n
 =\sum_{k\in\Z}\frac{b_k}{1-\lambda_k\eta}.
\end{equation}
There is no factorial in this series. Its convergence near zero follows from $|\lambda_k|\leq1/2$ and $\sum_k|b_k|<\infty$.

For $k\ne0$, the poles $p_k=1/\lambda_k$ are distinct and tend to infinity as $|k|\to\infty$. On any compact set avoiding these poles, the tail satisfies $|\lambda_k\eta|\leq1/2$, so it is normally convergent. Thus the right side of \eqref{eq:meromorphic} is meromorphic on $\C$ and zero on its pole complement by analytic continuation. Its residue at $p_k$ is $-b_k/\lambda_k$, whence $b_k=0$ for $k\ne0$. Finally $m_0=0$ gives $b_0=0$. This proves point continuation. Strip continuation follows by choosing any point inside the observed interval, using the continuous spatial representative.
\end{proof}

The same point proof works for $H^s$, $s>1/2$. Strip continuation also holds for $L^2$ data: pair the solution with a smooth function supported in the observed interval. The resulting Fourier coefficients are summable, and the residue argument recovers each coefficient times the corresponding test-function coefficient. Choosing a suitable test function for each Fourier index gives $f=0$.

\section{Observability, control, and linear stabilization}\label{sec:control}

\subsection{The energy duality and fixed actuation}

Use the Hilbert norm on $H^{-1}$ for which $$R:H^{-1}\to X$$ is an isometry. If the physical force is $a_th$, set $r=Rh$. The abstract equation is
\begin{equation}\label{eq:abstract}
\begin{cases}
   &u_t=Lu+B(t)r,\\
  &B(t)=Ra_tR^{-1},\\ 
  &B(t)^*v=a_tv,
  \end{cases}
\end{equation}
where the adjoint is taken in $X$. The last identity follows from
$$\langle B(t)r,v\rangle_X=\langle R^{-1}r,a_tv\rangle_{H^{-1},H^1}.$$
Smooth multiplication bounds $B(t)$ uniformly. The endpoint operator and its adjoint are
\begin{equation}\label{eq:endpoint}
 \begin{cases}
 &K_Tr=\int_0^T S(T-t)B(t)r(t)\,dt,\\
 &(K_T^*p)(t)=a_tS(t-T)p.
 \end{cases}
\end{equation}
For fixed $a_t=a$, Proposition \ref{thm:lin} gives $\ker K_T^*=\{0\}$, so $\Ran K_T$ is dense in $X$.

If $a$ has no zero, then $B=RaR^{-1}$ is invertible on $X$, and for $p\in X$ the control $$r(t)=T^{-1}B^{-1}S(t-T)p$$ gives $K_Tr=p$, thus $K_T$ is onto for every $T>0$.

Suppose now that $a(x_*)=0$ for some $x_*\in\T$. Fix a nonzero $\varphi\in C_c^\infty(-1,1)$ and, for large $n$, let $f_n=\varphi_n/\|\varphi_n\|_X$, where $\varphi_n(x)=\varphi(n(x-x_*))$ near $x_*$ and $\varphi_n=0$ elsewhere. Then $\|f_n\|_X=1$, $\|\partial_xf_n\|_2\leq1$, and $$\|f_n\|_2\leq n^{-1}\|\varphi\|_2/\|\varphi'\|_2\to0.$$ The symbol of $L$ gives
\begin{equation}\label{eq:Lsmoothing}
 \|Lf\|_X\leq\|f\|_2\qquad(f\in X),
\end{equation}
and hence
\begin{equation*}
 \sup_{|t|\leq T}\|(S(t)-I)f_n\|_X
 \leq T\|Lf_n\|_X\leq T\|f_n\|_2\longrightarrow0.
\end{equation*}
Since $f_n$ is supported in $\{|x-x_*|<1/n\}$ and $a(x_*)=0$,
\[
 \|af_n\|_X\leq\big(\|a\|_\infty+\|a'\|_\infty\big)\|f_n\|_2
 +\sup_{|x-x_*|\leq1/n}|a(x)|\longrightarrow0.
\]
Since multiplication by $a$ is bounded on $X$, it follows that $$(K_T^*S(T)f_n)(t)=aS(t)f_n$$ tends to zero in $L^2(0,T;X)$, while $\|S(T)f_n\|_X=1$. A bounded surjective operator between Hilbert spaces has an adjoint bounded below, so \eqref{eq:endpoint} cannot be surjective. Therefore, exact controllability fails. Null controllability also fails, since $S(T)$ is onto and null controllability would require $\Ran K_T$ to contain $S(T)X=X$. This proves the fixed-control assertions of Proposition \ref{thm:lincontrol}.

The forcing norm matters. With $h\in L^2(0,T;L^2)$, the endpoint map into $H^1$ is compact, so $$Ra_t:L^2\to H^2$$ is uniformly bounded, and $S(t)$ is bounded on $H^2$. The endpoint therefore maps bounded sets into bounded subsets of $H^2$, compactly embedded in $H^1$. Exact control in $H^1$ is impossible in that smaller forcing space, even with full spatial actuation.

\subsection{Moving observability}

We prove \eqref{eq:obs} under \eqref{eq:covering}. Since $$a_tf_x\text{ equals }
(a_tf)_x-(a_t)_xf,$$ coverage gives
\[
 \kappa_T\|f_x\|_2^2
 \leq2\int_0^T\|(a_tf)_x\|_2^2\,dt+2T\|a'\|_\infty^2\|f\|_2^2.
\]
By \eqref{eq:Lsmoothing},
$\|(S(t)-I)f\|_X\leq T\|f\|_2$ on $[0,T]$. Uniform multiplication bounds yield
\begin{equation}\label{eq:compactobs}
 \|f\|_X^2\leq C_T\left(
 \int_0^T\|a_tS(t)f\|_X^2\,dt+\|f\|_2^2\right).
\end{equation}
If \eqref{eq:obs} failed, there would be $f_n$ with unit $X$ norm and observations tending to zero. After extraction, $f_n$ converges weakly in $X$ and strongly in $L^2$. Applying \eqref{eq:compactobs} to differences gives strong convergence in $X$ to some $f$ of norm one, with $a_tS(t)f=0$. The nonempty open set $\{(x,t):0<t<T,\ a(x+ct)\ne0\}$ contains a space--time rectangle $\omega'\times(t_1,t_2)$. Proposition \ref{thm:lin}, applied to $S(t_1)f$, then gives $f=0$, a contradiction.

Replacing $f$ by $S(-T)p$ in \eqref{eq:obs} proves coercivity of $K_TK_T^*$. Consequently
\begin{equation*}
 P_T=K_T^*(K_TK_T^*)^{-1}:X\longrightarrow L^2(0,T;X)
\end{equation*}
is a bounded right inverse of $K_T$. Taking
$r=P_T(u_1-S(T)u_0)$ and $h=R^{-1}r$ proves the remaining part of Proposition \ref{thm:lincontrol}. These arguments preserve real-valuedness.

\subsection{Proof of linear stabilization}

For the fixed feedback, set $A=L-Ra^2$. This is a bounded compact operator on $X$, and its semigroup is contractive because
\begin{equation}\label{eq:fixedenergy}
 \frac12\frac{d}{dt}\|e^{tA}f\|_X^2
 =-\|ae^{tA}f\|_2^2.
\end{equation}
We first show weak convergence to zero. If $e^{t_nA}f\rightharpoonup v_0$ in $X$ for $t_n\to\infty$, then, on every bounded time interval, $e^{(t_n+t)A}f$ converges weakly to $e^{tA}v_0$. The tail of the integrable dissipation in \eqref{eq:fixedenergy} tends to zero. Weak lower semicontinuity gives $ae^{tA}v_0=0$. Thus this limiting trajectory solves the undamped linear equation \eqref{eq:linbbm} and vanishes on a strip, thus Proposition \ref{thm:lin} gives $v_0=0$. Every weak cluster point is zero, proving weak convergence of the trajectory.

The range of $A$ is dense. Indeed, $$A^*=-L-Ra^2,$$ if $A^*g=0$, its real energy pairing with $g$ gives $ag=0$, and then $Lg=0$. Hence $g$ is constant and, since $a\ne0$, $g=0$. For $g\in X$, compactness and commutation now give
\[
 e^{tA}Ag=Ae^{tA}g\longrightarrow0\quad\hbox{strongly in }X.
\]
Density of $\Ran A$ and contractivity imply strong convergence for every initial state. On the other hand, for any orthonormal sequence $g_n$ in $X$, compactness gives $Ag_n\to0$, and
\[
 \|(e^{tA}-I)g_n\|_X\leq t\|Ag_n\|_X\longrightarrow0.
\]
It follows that $\|e^{tA}\|_{X\to X}=1$ for every finite $t$. This proves the fixed part of Proposition \ref{thm:linstab}.

For the moving feedback, let $\Gamma(t)v=a_tv$, so that $\Gamma(t)=B(t)^*$ by \eqref{eq:abstract}. Equation \eqref{eq:linmovingdamp} becomes $u_t=Lu-\Gamma(t)^*\Gamma(t)u$, and
\begin{equation}\label{eq:movingenergy}
 \frac12\frac{d}{dt}\|u(t)\|_X^2=-\|\Gamma(t)u(t)\|_X^2.
\end{equation}
Choose $\tau=2\pi/|c|$. Observability on every interval $[s,s+\tau]$ has a constant independent of $s$: spatial translation conjugates $a(x+c(s+t))$ to $a(x+ct)$ and commutes with $S(t)$. If $v(t)=S(t-s)u(s)$, Duhamel's formula (note $\Gamma(t)^*=B(t)$ is uniformly bounded) gives
\[
 \sup_{s\leq t\leq s+\tau}\|v(t)-u(t)\|_X
 \leq \sup_t\|\Gamma(t)\|\sqrt\tau
 \left(\int_s^{s+\tau}\|\Gamma(t)u(t)\|_X^2dt\right)^{1/2}.
\]
Applying observability to $v$, followed by this bound, gives a constant $M_0\geq4$, independent of $s$, such that
\[
 \|u(s)\|_X^2\leq M_0\int_s^{s+\tau}\|\Gamma(t)u(t)\|_X^2dt.
\]
Together with \eqref{eq:movingenergy}, this yields
$$\|u(s+\tau)\|_X^2\leq\left(1-\frac{2}{M_0}\right)\|u(s)\|_X^2.$$ Iteration and contractivity between successive times prove the uniform exponential estimate.

\subsection{Local nonlinear moving control}

\begin{proof}[Proof of Proposition \ref{thm:nlcontrol}]
Let $Y=C([0,T];X)$ and define
\[
 Q(v)(t)=-\frac12\int_0^t S(t-s)R\partial_x(v(s)^2)\,ds.
\]
The algebra property gives
\begin{equation}\label{eq:quadratic}
\begin{cases}
& \|Q(v_1)-Q(v_2)\|_Y
 \leq C_Q(\|v_1\|_Y+\|v_2\|_Y)\|v_1-v_2\|_Y,
\\
 &\|Q(v)\|_Y\leq C_Q\|v\|_Y^2.
\end{cases}
\end{equation}
For $v\in Y$ choose the abstract control
$r_v=P_T(u_1-S(T)u_0-Q(v)(T))$, and set
\[
 \Phi(v)(t)=S(t)u_0+
 \int_0^t S(t-s)B(s)r_v(s)\,ds+Q(v)(t).
\]
By construction $\Phi(v)(0)=u_0$ and $\Phi(v)(T)=u_1$. Boundedness of $P_T$, uniform boundedness of $B$, and \eqref{eq:quadratic} imply
\[
 \|\Phi(v)\|_Y\leq C\big(\|u_0\|_X+\|u_1\|_X+\|v\|_Y^2\big),
\]
and the Lipschitz constant on a ball of radius $\rho$ is at most $C'\rho$. Choose $\rho$ proportional to $\|u_0\|_X+\|u_1\|_X$ and then take this sum sufficiently small. The map preserves the ball and is a strict contraction. Its fixed point $u$ solves \eqref{eq:nlmoving} in the mild sense with $h=R^{-1}r_u$. The same estimates give the stated state and control bound. For zero endpoints the zero solution and control suffice.
\end{proof}

The change of coordinates $v(x,t)=u(x-ct,t)$ converts \eqref{eq:nlmoving} into
\[
 v_t-v_{xxt}-cv_{xxx}+(1+c)v_x+vv_x=a(x)h(x-ct,t),
\]
the moving-frame KdV--BBM formulation used in \cite{RosierZhang}. Their global exact-control and semiglobal nonlinear stabilization theorems require further nonlinear arguments. Conversely, failure of a bounded linear right inverse for a fixed actuator with a gap prevents the preceding perturbation proof at zero, so it does not by itself prove general nonlinear uncontrollability.

\section{Nonlinear weak stabilization under fixed damping}\label{sec:nonlineardamping}

\begin{proof}[Proof of Theorem \ref{thm:nldamp}]
The operators $RD_j$, for $j=0,1$, are bounded on $H^1$, so \eqref{eq:nldamp} is a locally Lipschitz ODE there. The energy identity is
\begin{equation}\label{eq:nldampenergy}
 \frac12\frac{d}{dt}\|u(t)\|_{H^1}^2
 =-\|a\partial_x^ju(t)\|_2^2.
\end{equation}
For $j=1$ the pairing is between $H^{-1}$ and $H^1$. This identity gives global existence and a uniform $H^1$ bound; the equation then bounds $u_t$ uniformly in $H^1$. Integration in space conserves the mean for $j=1$.

We follow the compactness argument of \cite[Corollary~4.2]{RosierZhang} and \cite{Rosier2016}. Let $t_n\to\infty$ with $u(t_n)\rightharpoonup v_0$ in $H^1$, and choose $1/2<\sigma<1$. Compact Sobolev embedding and the derivative bound give, after extraction,
\begin{equation*}
 \begin{cases}
 u(t_n+\cdot)\longrightarrow v
       &\text{in }C([0,1];H^\sigma),\\
 u(t_n+\cdot)\rightharpoonup v
       &\text{in }L^2(0,1;H^1).
 \end{cases}
\end{equation*}
Both $RD_j$ are bounded on $H^\sigma$, and $H^\sigma$ is an algebra. Passing to the integral equation identifies $v$ as the damped solution with initial value $v_0$. Uniqueness of that ODE in $H^\sigma$ identifies it with the $H^1$ solution. The dissipation in \eqref{eq:nldampenergy} is integrable on $[0,\infty)$; hence its integral on $[t_n,t_n+1]$ tends to zero. Weak lower semicontinuity yields $a\partial_x^jv=0$ on $\T\times(0,1)$. Thus $D_jv=0$ and $v$ solves the conservative equation on this interval.

Let $\bar v\in C(\R;H^1)$ be the solution of the conservative equation \eqref{eq:bbm} with $\bar v(0)=v_0$. Since $D_jv=0$ on $[0,1]$, uniqueness for \eqref{eq:bbm} gives $\bar v=v$ on $[0,1]$, so Theorem \ref{thm:main} may be applied to $\bar v$. For $j=0$, the limit vanishes on a strip where $a\ne0$, and Theorem \ref{thm:main} gives $\bar v\equiv0$, hence $v_0=0$. For $j=1$, we have $v_x=0$ on such a strip. There $\partial_x(v+v^2/2)=0$ and $v_{xxt}=0$ in distributions, so the undamped equation \eqref{eq:bbm} gives $v_t=0$. Hence $v$ equals one constant $C$ on a space--time rectangle. The constant-value conclusion of Theorem \ref{thm:main}, applied to $\bar v$, gives $v_0\equiv C$. Since $u(t_n)\rightharpoonup v_0$ in $H^1$ and the mean of $u(t)$ is conserved, $C=m_1$.

Every weak cluster point of the bounded trajectory is therefore $m_j$. Weak sequential compactness gives weak convergence of the entire trajectory in $H^1$, and compact embedding gives strong convergence in each $H^s$, $s<1$.
\end{proof}

Strong nonlinear $H^1$ attraction and quantitative decay remain beyond this argument. The linear density-and-commutation proof does not transfer to a nonlinear flow, and qualitative continuation does not remove the fixed-actuator observability obstruction.

\subsection{Boundary feedback}
Let $\ell>0$ and $\mu_0,\mu_\ell\in\R$. Rosier \cite[Section~2.3]{Rosier2016} considers the BBM equation on the interval $(0,\ell)$ with nonhomogeneous boundary conditions of feedback type,
\begin{equation}\label{eq:boundarysys}
 \begin{cases}
 u_t-u_{txx}+u_x+uu_x=0&0<x<\ell,\ t\geq0,\\
 u_{tx}(0,t)=\mu_0u(0,t)+\tfrac13u(0,t)^2,& t\geq0,\\
 u_{tx}(\ell,t)=\mu_\ell u(\ell,t)+\tfrac13u(\ell,t)^2,& t\geq0,\\
 \\ u(x,0)=u_0,&0<x<\ell.
 \end{cases}
\end{equation}
Multiplying by $u$ gives, at least formally,
\[
 \frac12\frac{d}{dt}\|u(t)\|_{H^1(0,\ell)}^2=\Big(\mu_\ell-\frac12\Big)u(\ell,t)^2+\Big(\frac12-\mu_0\Big)u(0,t)^2,
\]
so that the $H^1$ norm is nonincreasing when $\mu_0>\frac12>\mu_\ell$. For $u_0\in H^1(0,\ell)$, \cite[Theorem~4]{Rosier2016} provides a unique global solution $u\in C([0,\infty);H^1(0,\ell))$, and it shows that $u(t)\to0$ weakly in $H^1(0,\ell)$ provided the unique continuation conjecture holds. The limiting trajectories are not periodic, so Theorem~\ref{thm:main} does not apply to them directly, but Proposition~\ref{prop:liouville} does.

\begin{corollary}[Boundary feedback]\label{cor:boundary}
Let $\ell>0$, $\mu_0>\frac12>\mu_\ell$, and $u_0\in H^1(0,\ell)$. Then the solution of \eqref{eq:boundarysys} satisfies, as $t\to\infty$,
\[
 u(t)\rightharpoonup0\ \text{ weakly in }H^1(0,\ell),\qquad
 \|u(t)\|_{H^s(0,\ell)}\to0\quad(s<1).
\]
There is no size restriction on $u_0$.
\end{corollary}

\begin{proof}
We follow the proof of \cite[Theorem~4]{Rosier2016}, in which the conjecture is used only in the last step. Let $t_n\to\infty$ with $u(t_n)\rightharpoonup v_0$ in $H^1(0,\ell)$, fix $T>0$, and let $v\in C([0,T];H^1(0,\ell))$ be the solution of \eqref{eq:boundarysys} with initial value $v_0$. By \cite[(2.29)]{Rosier2016}, $v_t\in C([0,T];H^2(0,\ell))$. The energy identity and the continuity of the flow give, as in \cite[(2.31)]{Rosier2016},
\[
 \Big(\frac12-\mu_\ell\Big)\int_0^T v(\ell,t)^2\,dt+\Big(\mu_0-\frac12\Big)\int_0^Tv(0,t)^2\,dt=0 .
\]
Hence $v(0,t)=v(\ell,t)=0$ for $t\in[0,T]$. Consequently $v_t(0,t)=v_t(\ell,t)=0$, since point evaluations are continuous on $H^1(0,\ell)$, and the boundary conditions give $v_{tx}(0,t)=v_{tx}(\ell,t)=0$. Let $\widetilde v$ be the extension of $v$ by zero to $\R$. Since $v$ vanishes at the endpoints, $\widetilde v\in C([0,T];H^1(\R))$. Moreover, $\widetilde v$ solves \eqref{eq:bbm} in $\mathcal D'(\R\times(0,T))$. Indeed, when the equation, which holds almost everywhere in $(0,\ell)\times(0,T)$, is integrated against a test function, every boundary term produced at $x=0$ or $x=\ell$ contains one of the factors $v_{tx}$, $v_t$ or $v+v^2/2$ evaluated at that endpoint, and these vanish. By Section~\ref{sec:flow} (with $\beta=1$), $\widetilde v$ is, on $[0,T]$, a trajectory of the global $H^1(\R)$ flow of \eqref{eq:beta}. Let $z\in C(\R;H^1(\R))$ be this trajectory. Then $z$ is supported in $[0,\ell]$ for $t\in(0,T)$, and Proposition~\ref{prop:liouville} gives $z\equiv0$. In particular $v_0=0$. Every weak limit point of the bounded trajectory $(u(t))_{t\geq0}$ in $H^1(0,\ell)$ is therefore zero, which proves the weak convergence; the compact embedding $H^1(0,\ell)\subset H^s(0,\ell)$ for $s<1$ gives the strong convergence.
\end{proof}

The extension by zero in the proof of \cite[Theorem~4]{Rosier2016} produces a solution on the line which is supported in $[0,\ell]$ during a time interval. For every $\ell>0$, this is exactly the situation of Proposition~\ref{prop:liouville}, and no comparison between $\ell$ and the period $2\pi$ is needed.

\section{Further remarks and open questions}\label{sec:remarks}

\begin{remark}[A quadratic virial functional]\label{rem:virial}
Replacing the weight $e^{\alpha x}$ by $x$ and the Hamiltonian density by the indefinite density $u^2-u_x^2$ gives another exact identity on the fundamental domain: for $$V(t)=\frac12\int_{-\pi}^{\pi}x\,(u^2-u_x^2)\,dx$$ and $q$, $\psi$ as in \eqref{eq:defpsi},
\[
 V'(t)=\frac12\|\Lambda q\|_2^2+\mathcal E(u)+2\pi\Big(u(\pi)\psi_{xx}(\pi)-\psi(\pi)^2+\psi_x(\pi)^2+\frac12u(\pi)^2+\frac13u(\pi)^3\Big),
\]
where $\Lambda=I-2R$ is the Fourier multiplier with symbol $(k^2-1)/(k^2+1)$ and
$$\mathcal E(u)=-\frac13\int_\T u^3-\frac18\int_\T u^4+\frac12\int_\T u^2\psi+\int_{-\pi}^{\pi}x\,u^2\psi_x.$$ The linear part is a Kato-type monotonicity formula \cite{Kato}: the symbol $1-\xi^2$ of the density has the sign of the group velocity $$\frac{d}{d\xi}\frac{\xi}{1+\xi^2}=\frac{1-\xi^2}{(1+\xi^2)^2}$$ of the linearized equation. The kernel of $\Lambda$ consists of the modes $\pm1$, which are controlled by the two conditions $\psi(\pi)=\psi_x(\pi)=0$. However, in our estimates the cubic and quartic remainder $\mathcal E(u)$ can be absorbed only for small data. For $C=0$ and $\|u_0\|_{H^1}$ small, this gives $V'\geq c\|u\|_2^2$ with $c>0$ under \eqref{eq:pointobs} with $x_0=\pi$ (hence, by Lemma~\ref{lem:allt}, for all $t$), and the conservation of $$\int_\T(3u^2+u^3)\,dx$$ keeps $\|u(t)\|_2$ away from zero if $u_0\neq0$, which contradicts the boundedness of $V$. This yields Theorem~\ref{thm:point} with $C=0$ for small data only, thus the exponentially weighted Hamiltonian avoids any smallness.
\end{remark}

\begin{remark}[Open questions]
We have below several open problems:
\begin{itemize}
\item[(i)] Does the observation of $u$ alone at one point suffice for \eqref{eq:bbm} at the levels $C\neq-1$, in particular at $C=0$? It does for the linearized equation (Proposition~\ref{thm:lin}), and it does not at the level $C=-1$ (Remark~\ref{rem:level}).
\item[(ii)] Does Theorem~\ref{thm:main} hold in $H^s(\T)$ for $1/2\leq s<1$? It fails for $0\leq s<1/2$, even for small data, by the step profiles \eqref{eq:step}. For $1/2<s<1$ the solutions are still continuous in $x$, but the uniform bounds and the compactness at both ends of time used above rest on the conservation of the $H^1$ norm.
\end{itemize}
\end{remark}

\begin{remark}[Numerical checks]\label{rem:numerics}
The identity \eqref{eq:fluxidentity} was checked numerically for random real trigonometric polynomials $u$ (neither small nor vanishing near the cut), several levels $C$ and several values of $\alpha$, with Gauss--Legendre quadrature on $[-\pi,\pi]$ (relative agreement to at least $10^{-10}$ in double precision and to about $10^{-39}$ in $40$-digit arithmetic), and along numerically computed solutions of \eqref{eq:bbm} (pseudo-spectral discretization), by comparing $\widetilde N_\alpha(t)-\widetilde N_\alpha(0)$ with the time integral of the right-hand side of \eqref{eq:fluxidentity}. The identity of Remark~\ref{rem:virial} and the symmetry of Remark~\ref{rem:level} were checked in the same way, and the dichotomy of Proposition~\ref{thm:lincontrol} between a fixed actuator with a zero and a moving actuator was illustrated through the smallest eigenvalues of truncated observability Gramians. The Python scripts are available from the authors upon request. These computations corroborate, but do not replace, the proofs.
\end{remark}

\subsection*{Acknowledgments}
This work began during visits of the third author to the Departamento de Matem\'atica of the Universidade Federal de Pernambuco (UFPE) in 2025, and he thanks the department for its hospitality. The second author thanks the Laboratoire Analyse, G\'eom\'etrie et Applications (LAGA), Universit\'e Sorbonne Paris Nord, where he spent a year as a postdoctoral researcher, for its hospitality.


\begin{thebibliography}{99}
\bibitem{BautistaPazoto}
G. J. Bautista and A. F. Pazoto,
\href{https://doi.org/10.1007/s10884-018-9689-4}{\emph{Large-time behavior of a linear Boussinesq system for the water waves}},
J. Dynam. Differential Equations \textbf{31} (2019), 959--978.

\bibitem{BBM}
T. B. Benjamin, J. L. Bona, and J. J. Mahony,
\href{https://doi.org/10.1098/rsta.1972.0032}{\emph{Model equations for long waves in nonlinear dispersive systems}},
Philos. Trans. Roy. Soc. London Ser. A \textbf{272} (1972), 47--78.

\bibitem{BonaTzvetkov}
J. L. Bona and N. Tzvetkov,
\href{https://doi.org/10.3934/dcds.2009.23.1241}{\emph{Sharp well-posedness results for the BBM equation}},
Discrete Contin. Dyn. Syst. \textbf{23} (2009), no. 4, 1241--1252.

\bibitem{Bourgain}
J. Bourgain,
\href{https://doi.org/10.1155/S1073792897000305}{\emph{On the compactness of the support of solutions of dispersive equations}},
Internat. Math. Res. Notices \textbf{1997}, no. 9, 437--447.

\bibitem{CapistranoGomes}
R. de A. Capistrano-Filho and A. Gomes,
\href{https://doi.org/10.1051/cocv/2022085}{\emph{Global control aspects for long waves in nonlinear dispersive media}},
ESAIM Control Optim. Calc. Var. \textbf{29} (2023), Paper No. 7.

\bibitem{CapistranoKwakVielma}
R. de A. Capistrano-Filho, C. Kwak, and F. J. Vielma Leal,
\href{https://www.sciencedirect.com/science/article/pii/S1468121822001043}{\emph{On the control issues for higher-order nonlinear dispersive equations on the circle}},
Nonlinear Anal. Real World Appl. \textbf{68} (2022), Paper No. 103695.

\bibitem{daSilvaFreire}
P. L. da Silva and I. L. Freire,
\href{https://doi.org/10.1007/s00605-020-01453-0}{\emph{A geometrical demonstration for continuation of solutions of the generalised BBM equation}},
Monatsh. Math. \textbf{194} (2021), 495--502.

\bibitem{DavilaMenzala}
M. D\'avila and G. Perla Menzala,
\href{https://doi.org/10.1007/s000300050051}{\emph{Unique continuation for the Benjamin--Bona--Mahony and Boussinesq's equations}},
NoDEA Nonlinear Differential Equations Appl. \textbf{5} (1998), no. 3, 367--382.

\bibitem{DehmanGerardLebeau}
B. Dehman, P. G\'erard, and G. Lebeau,
\href{https://doi.org/10.1007/s00209-006-0005-3}{\emph{Stabilization and control for the nonlinear Schr\"odinger equation on a compact surface}},
Math. Z. \textbf{254} (2006), no. 4, 729--749.

\bibitem{HongPonce}
C. Hong and G. Ponce,
\href{https://doi.org/10.1016/j.jde.2024.02.021}{\emph{On special properties of solutions to the Benjamin--Bona--Mahony equation}},
J. Differential Equations \textbf{393} (2024), 321--342.

\bibitem{Kato}
T. Kato,
\emph{On the Cauchy problem for the (generalized) Korteweg--de Vries equation},
in: Studies in Applied Mathematics, Adv. Math. Suppl. Stud. \textbf{8}, Academic Press, New York, 1983, 93--128.

\bibitem{KPV02}
C. E. Kenig, G. Ponce, and L. Vega,
\href{https://doi.org/10.1016/S0294-1449(01)00073-7}{\emph{On the support of solutions to the generalized KdV equation}},
Ann. Inst. H. Poincar\'e Anal. Non Lin\'eaire \textbf{19} (2002), no. 2, 191--208.

\bibitem{KPV03}
C. E. Kenig, G. Ponce, and L. Vega,
\href{https://doi.org/10.4310/MRL.2003.v10.n6.a10}{\emph{On the unique continuation of solutions to the generalized KdV equation}},
Math. Res. Lett. \textbf{10} (2003), no. 6, 833--846.

\bibitem{LaurentLinaresRosier}
C. Laurent, F. Linares, and L. Rosier,
\href{https://doi.org/10.1007/s00205-015-0887-5}{\emph{Control and stabilization of the Benjamin--Ono equation in $L^2(\mathbb T)$}},
Arch. Ration. Mech. Anal. \textbf{218} (2015), no. 3, 1531--1575.

\bibitem{LaurentRosierZhang}
C. Laurent, L. Rosier, and B.-Y. Zhang,
\href{https://doi.org/10.1080/03605300903585336}{\emph{Control and stabilization of the Korteweg--de Vries equation on a periodic domain}},
Comm. Partial Differential Equations \textbf{35} (2010), no. 4, 707--744.

\bibitem{Mammeri}
Y. Mammeri,
\href{https://doi.org/10.57262/die/1356019780}{\emph{Unique continuation property for the KP--BBM-II equation}},
Differential Integral Equations \textbf{22} (2009), no. 3-4, 393--399.

\bibitem{Micu}
S. Micu,
\href{https://doi.org/10.1137/S0363012999362499}{\emph{On the controllability of the linearized Benjamin--Bona--Mahony equation}},
SIAM J. Control Optim. \textbf{39} (2001), no. 6, 1677--1696.

\bibitem{MORZ}
S. Micu, J. H. Ortega, L. Rosier, and B.-Y. Zhang,
\href{https://doi.org/10.3934/dcds.2009.24.273}{\emph{Control and stabilization of a family of Boussinesq systems}},
Discrete Contin. Dyn. Syst. \textbf{24} (2009), no. 2, 273--313.

\bibitem{MicuPazoto}
S. Micu and A. F. Pazoto,
\href{https://doi.org/10.1007/s11854-018-0074-3}{\emph{Stabilization of a Boussinesq system with localized damping}},
JAMA \textbf{137} (2019), 291--337.

\bibitem{MontesCordoba}
A. M. Montes and R. C\'ordoba,
\href{https://doi.org/10.1007/s00605-025-02115-9}{\emph{Exact controllability for the regularized Benjamin--Ono equation on a periodic domain}},
Monatsh. Math. \textbf{208} (2025), 347--368.

\bibitem{PazotoSotoVieira}
A. F. Pazoto and M. D. Soto Vieira,
\href{https://doi.org/10.1051/cocv/2023040}{\emph{Unique continuation and time decay for a higher-order water wave model}},
ESAIM Control Optim. Calc. Var. \textbf{29} (2023), Paper No. 49.

\bibitem{Peregrine}
D. H. Peregrine,
\href{https://doi.org/10.1017/S0022112066001678}{\emph{Calculations of the development of an undular bore}},
J. Fluid Mech. \textbf{25} (1966), 321--330.

\bibitem{Rosier2016}
L. Rosier,
\href{https://doi.org/10.4208/jms.v49n2.16.06}{\emph{On the Benjamin--Bona--Mahony equation with a localized damping}},
J. Math. Study \textbf{49} (2016), no. 2, 195--204.

\bibitem{RosierZhang}
L. Rosier and B.-Y. Zhang,
\href{https://doi.org/10.1016/j.jde.2012.08.014}{\emph{Unique continuation property and control for the Benjamin--Bona--Mahony equation on a periodic domain}},
J. Differential Equations \textbf{254} (2013), no. 1, 141--178.

\bibitem{Roumegoux}
D. Roum\'egoux,
\href{https://intlpress.com/api/bgcloud-front/resource/pdf/volume/1805802058919854082-1805802058919854082-ad56d8e10b8ef411fe67455f150577d5.pdf}{\emph{A symplectic non-squeezing theorem for BBM equation}},
Dyn. Partial Differ. Equ. \textbf{7} (2010), no. 4, 289--305.

\bibitem{RussellZhang}
D. L. Russell and B.-Y. Zhang,
\href{https://www.jstor.org/stable/2155248}{\emph{Exact controllability and stabilizability of the Korteweg--de Vries equation}},
Trans. Amer. Math. Soc. \textbf{348} (1996), no. 9, 3643--3672.

\bibitem{SautScheurer}
J.-C. Saut and B. Scheurer,
\href{https://doi.org/10.1016/0022-0396(87)90043-X}{\emph{Unique continuation for some evolution equations}},
J. Differential Equations \textbf{66} (1987), no. 1, 118--139.

\bibitem{SierraFonsecaPazoto}
O. A. Sierra Fonseca and A. F. Pazoto,
\href{https://doi.org/10.1007/s00033-022-01861-2}{\emph{Asymptotic behavior of a linear higher-order BBM-system with damping}},
Z. Angew. Math. Phys. \textbf{73} (2022), Paper No. 231.

\bibitem{Yamamoto}
M. Yamamoto,
\href{https://www.degruyterbrill.com/document/doi/10.1515/156939403770888264}{\emph{One unique continuation for a linearized Benjamin--Bona--Mahony equation}},
J. Inverse Ill-Posed Probl. \textbf{11} (2003), 537--543.

\bibitem{Zhang92}
B.-Y. Zhang,
\href{https://doi.org/10.1137/0523004}{\emph{Unique continuation for the Korteweg--de Vries equation}},
SIAM J. Math. Anal. \textbf{23} (1992), no. 1, 55--71.

\bibitem{ZhangZuazua}
X. Zhang and E. Zuazua,
\href{https://doi.org/10.1007/s00208-002-0391-8}{\emph{Unique continuation for the linearized Benjamin--Bona--Mahony equation with space-dependent potential}},
Math. Ann. \textbf{325} (2003), no. 3, 543--582.
\end{thebibliography}
\end{document}